\documentclass[a4paper,12pt]{article}
\usepackage{amssymb,amsthm,latexsym,amsmath,url}

\newtheorem{lem}{Lemma}[section]
\newtheorem{dfn}[lem]{Definition}
\newtheorem{thm}[lem]{Theorem}
\newtheorem{cor}[lem]{Corollary}
\newtheorem{pro}[lem]{Proposition}

\newtheorem{exa}[lem]{Example}

\def\B{\mathcal{GB}_2}
\def\T{\mathcal{GT}\!_2}
\def\l{\langle}
\def\r{\rangle}

\title{\normalsize\bf A GENERALIZATION OF 2-BAER GROUPS}
\author{\small{\textsc{Luise-Charlotte Kappe}}\\
\small{Department of Mathematical Sciences, Binghamton University}\\
\small{Binghamton, NY 13902-6000, USA}\\
\small{E-mail: menger@math.binghamton.edu}\\
[10pt]
\small{\textsc{Antonio Tortora}}\\
\small{Dipartimento di Matematica, Universit\`a di Salerno}\\
\small{Via Giovanni Paolo II, 132 - 84084 - Fisciano (SA), Italy}\\
\small{E-mail: antortora@unisa.it}}
\date{}

\begin{document}
\maketitle

\begin{abstract}
\noindent
A group in which all cyclic subgroups are $2$-subnormal is called a $2$-Baer group. The topic of this paper are generalized 2-Baer groups, i.e. groups in which the non-2-subnormal cyclic subgroups generate a proper subgroup of the group. If this subgroup is non-trivial, the group is called a generalized $T_2$-group. In particular, we provide structure results for such groups, investigate their nilpotency class and construct examples of finite $p$-groups which are generalized $T_2$-groups.\\

\end{abstract}

\section{Introduction}

Let $n$ be a positive integer. A subgroup $H$ of a group $G$ is called $n$-subnormal, denoted by $H\lhd_n G$, if there exist distinct subgroups $H_0=H,H_1,\ldots, H_n=G$ such that
$$H=H_0\lhd H_1\lhd\ldots\lhd H_n=G.$$
In a group of nilpotency class $n$, all subgroups are $n$-subnormal. Conversely, by a well-known result of Roseblade \cite{Ros}, a group with all subgroups $n$-subnormal is nilpotent of class bounded by a function of $n$. For $n=1$, having all subgroups $n$-subnormal, hence normal, is equivalent to having all cyclic subgroups $n$-subnormal; but this is no longer the case if $n\geq 2$ \cite{Ma}.

A group $G$ is called an {\em $n$-Baer group} if all of its cyclic subgroups are $n$-subnormal. It can be easily seen that every $n$-Baer group is $(n+1)$-Engel, i.e. $[x,_{n+1} y]=1$ for all $x,y\in G$, where
$$[x,_1 y]=[x,y]=x^{-1}x^{y}\quad {\rm and} \quad [x,_k y]=[[x,_{k-1} y], y],\; k\geq 2.$$
We say $x\in G$ is a {\em right $n$-Engel element}, if $[x,_n g]=1$ for all $g\in G$, and a {\em left $n$-Engel element}, if $[g,_n x]=1$ for all $g\in G$.

Of course, the class of 1-Baer groups coincides with the familiar Dedekind groups. For the finite case these groups were classified by Dedekind in 1897 in \cite{De}, and for the case of infinite groups by Baer in \cite{Ba}. Our interest in a generalization of $n$-Baer groups, in particular $2$-Baer groups, is motivated by D.\,Cappitt's research in a generalization of Dedekind groups \cite{Ca}. He considers groups in which the non-normal subgroups of a group generate a proper subgroup of the group. Such groups were called {\em generalized Dedekind groups}. In \cite{Re} it was shown that the subgroup generated by all non-normal cyclic subgroups coincides with the subgroup generated by all non-normal subgroups. Generalized Dedekind groups which are not Dedekind groups are called {\em generalized Hamiltonian groups}.

Recalling that a Dedekind group is an $n$-Baer group for $n=1$, we want to extend Cappitt's concept of generalized Dedekind groups for $n>1$. For any group $G$, let
$$T_n(G)=\langle x\in G \,|\, \langle x\rangle \ntriangleleft_n G\rangle.$$
If all cyclic subgroups are $n$-subnormal in $G$, i.e. $G$ is an $n$-Baer group, we define $T_n(G)=1$. Notice that $T_n(G)$ is a characteristic subgroup of $G$ and $G/T_n(G)$ is an $n$-Baer group.

In light of \cite{Ca} and \cite{KR}, we generalize the concept of an $n$-Baer group as follows.

\begin{dfn}
If $G$ is a group with $T_n(G)\neq G$, then $G$ is called a generalized $n$-Baer group, and, if in addition $T_n(G)$ is non-trivial, then
$G$ is called a generalized $T_n$-group.
\end{dfn}

\noindent The class of generalized $n$-Baer groups and the class of generalized $T_n$-groups will be denoted by ${\mathcal{GB}_n}$ and ${\mathcal{GT}\!_n}$, respectively.

In \cite{Ca} it was shown that a group in ${\mathcal{GB}_1}$ is either abelian or torsion and nilpotent of class 2. The question arises, if there are any restrictions on the nilpotency class of generalized $n$-Baer groups. It was shown in \cite{He} and \cite{Mah} that 2-Baer groups are of nilpotency class at most 3. However, as it was shown in \cite{GK}, $n$-Baer groups with $n\geq 3$ are not necessarily nilpotent. Thus it appears to be natural to restrict our attention to generalized 2-Baer groups. As we will see, our results are, for the most part, analogue to those for generalized Dedekind groups, i.e. generalized 1-Baer groups. Here are some of the details of our investigations.

For a non-torsion group, which is a generalized 2-Baer group, we will show in Theorem \ref{nontorsion} that such a group is actually 2-Baer or, equivalently, 2-Engel. This is the direct analogue of Cappitt's result in \cite{Ca}. For torsion groups which are generalized 1-Baer groups, Cappitt \cite{Ca} shows that such groups are either 1-Baer groups or a direct product of a $p$-group in ${\mathcal{GB}_1}$ and another factor which is 1-Baer without $p$-torsion. Our result, presented in Theorem \ref{main}, is exactly the same, replacing 1-Baer by 2-Baer, and ${\mathcal{GB}_1}$ by $\B$.

However when it comes to restrictions on the nilpotency class of torsion groups in $\T$, our result differs from the one obtained by Cappitt in \cite{Ca}, where it was shown that torsion groups in ${\mathcal{GT}\!_1}$ are of nilpotency class 2. For metabelian $p$-groups in $\T$ with $p>2$ we can show that such groups are nilpotent of class exactly 3 (Theorem \ref{metp}). In Example \ref{2-group} we present a metabelian 2-group obtained by {\sf GAP} \cite{GAP}, which is in $\T$ but of nilpotency class 4, showing that Theorem \ref{metp} cannot be extended to 2-groups.  In Example \ref{ex}, for any prime $p$, we provide metabelian $p$-groups which are in $\T$ and have nilpotency class 3.

Using results by Endimioni in \cite{En89,En91,En93}, it can be shown that locally finite $p$-groups in $\T$ are soluble if $p\neq 5$, and for $p>5$ those groups are nilpotent of bounded class (see Theorem \ref{End}). The exact bound of the nilpotency class for groups in $\T$ remains an open question.

\section{Structure Results}

In this section we provide some structure results for generalized 2-Baer groups. As already pointed out in the introduction, the results for non-torsion and torsion groups alike are analogue to those determined by Cappitt \cite{Ca} for the case of generalized 1-Baer groups. We start with some preparatory lemmas.

\begin{lem}\label{normal} Let $G\in\B$.
\begin{itemize}
\item[$(i)$] If $x\in G\backslash T_2(G)$ and $H$ is a subgroup of $G$ containing $x$, then $H\in\B$.
\item[$(ii)$] If $N$ is a normal subgroup of $G$ contained in $T_2(G)$, then $G/N\in\B$.
\end{itemize}
\end{lem}

\begin{proof}
Since $T_2(H)\leq T_2(G)$ and $T_2(G/N)\leq T_2(G)/N$, we have $x\in H\backslash T_2(H)$ and $T_2(G/N)\neq G/N$, respectively.
\end{proof}

Given an element $x$ of a group $G$, throughout the paper, we denote by $\l x^G\r$ the normal closure of $x$ in $G$, i.e. the smallest normal subgroup of $G$ containing~$x$.

\begin{lem}\label{closure}
Let $G\in\B$ and $x\in G\backslash T_2(G)$. Then $\langle x^G\rangle$ is nilpotent of class at most $2$. In particular, $x$ is a left $3$-Engel element.
\end{lem}

\begin{proof}
Clearly, $\langle x\rangle\lhd\langle x^G\rangle$ and $N_G(\l x\r)/C_G(x)$ is abelian. Denote by $H$ the commutator subgroup of $\l x^G\r$. Then $\l x^G\r C_G(x)/C_G(x)$ is abelian and therefore $H\leq C_G(x)$. Let $a\in H$ and $g\in G$. Now $a^{g^{-1}}\in H$, so that $[a^{g^{-1}},x]=1$ or $[a,x^g]=1$. This means that $H\leq C_G(x^g)$, for all $g\in G$. Hence, $H\leq Z(\langle x^G\rangle)$ and $\langle x^G\rangle$ is nilpotent of class $\leq 2$.
\end{proof}

\begin{pro}\label{gen}
Let $G\in\B$. Then, for any $d\geq 1$, every $d$-generator subgroup of $G$ is nilpotent of class at most $2(d+1)$.
\end{pro}

\begin{proof}
Let $H=\langle h_1,\ldots,h_d\rangle$ be a subgroup of $G$ and $x\in G\backslash T_2(G)$. Consider $K=\langle h_1,\ldots,h_d,x\rangle$. Next we want to show that we can find elements $k_1,\ldots,k_d$ such that $K=\l k_1,\ldots,k_d,x\r$ and $k_i\in G\backslash T_2(G)$. For any $1\leq i\leq d$ we set $k_i=h_i$, if $h_i\in G\backslash T_2(G)$ and $k_i=h_i x$ otherwise. So each $k_i\in G\backslash T_2(G)$ and thus $K$ is contained in the subgroup $\l k_{1}^G\rangle\ldots\langle k_{d}^G\r\l x^G\r$, which is nilpotent of class $\leq 2(d+1)$ by Lemma \ref{closure}.
\end{proof}

\begin{lem}\label{infinite}
Let $G\in\B$ and $x$ an element of $G\backslash T_2(G)$ of infinite order. Then:
\begin{itemize}
\item[$(i)$] For any $g\in G$, there exists an integer $n=n(g)\geq 0$ such that $x^n g$ is an element of $G\backslash T_2(G)$ of infinite order;
\item[$(ii)$] $\langle x^G\rangle$ is abelian and $x$ is a right $2$-Engel element.
\end{itemize}
\end{lem}

\begin{proof}
$(i)$ Given an arbitrary element $g$ of $G$, we consider first the case that $g\in T_2(G)$. Thus $xg\in G\backslash T_2(G)$. Moreover, by Lemma \ref{closure}, $\l x,g\r\leq\l x^G\r\l (xg)^G\r$ and hence $\l x,g\r$ is nilpotent. If $xg$ has finite order, this implies that $x^2 g$ and $x^3 g$ have infinite order. However, if $x^2 g\in T_2(G)$, then $x^3 g\in G\backslash T_2(G)$. Next let $g\in G\backslash T_2(G)$. Obviously, we may assume that $g$ has finite order. If
$xg\in G\backslash T_2(G)$, our claim follows as before. Thus we may assume $xg\in G\backslash T_2(G)$. By Lemma \ref{closure} it follows that $\l x,g\r$ is nilpotent and so $xg$ has infinite order.\\

$(ii)$ We prove that $x\in Z(\l x^G\r)$. Suppose to the contrary and let $a\in\l x^{G}\r$ such that $x^a\neq x$. Since $\l x\r\lhd \l x^{G}\r$, we have $x^a=x^m$ for some integer $m\neq 0,1$. Similarly, $x^{a^{-1}}=x^{m'}$ with $m'\neq 0,1$. Then $x^{mm'}=x$, or $x^{mm'-1}=1$. This gives $m=m'=-1$ and so $x^a=x^{-1}$. We conclude $[a,x]=x^2$. Lemma \ref{closure} implies that $x^2\in Z(\langle x^G\rangle)$ and therefore $1=[a,x^2]=[a,x]^2=x^4$, a contradiction. Hence, $x\in Z(\l x^G\r)$ and $\langle x^G\rangle$ is abelian. In particular, $[g,x,x]=1$ for all $g\in G$. On the other hand, by $(i)$, for any $g\in G$ there exists $k\geq 0$ such that $x^k g$ is an element of $G\backslash T_2(G)$ of infinite order. Thus $1=[x,x^k g,x^k g]=[x,g,x^k g]=[x,g,g]$, which shows that $x$ is a right 2-Engel element.
\end{proof}

In \cite{He}, Heineken already showed that in the class of non-torsion groups, 2-Baer and 2-Engel are equivalent properties. We extend this result here for generalized 2-Baer groups.

\begin{thm}\label{nontorsion} Let $G$ be a non-torsion group. Then the following are equivalent:
\begin{itemize}
\item[$(i)$] $G$ is a $2$-Baer group;
\item[$(ii)$] $G$ is a generalized $2$-Baer group;
\item[$(iii)$] $G$ is a $2$-Engel group.
\end{itemize}
\end{thm}

\begin{proof}
Denote by
$$R_2(G)=\{x\in G\,|\,[x,g,g]=1\;{\rm for \; all}\; g\in G\}$$
the set of all right 2-Engel elements of $G$.
Since the normal closure of any $2$-Engel element is abelian \cite{Le} (see also \cite[Theorem 7.13]{Ro2}), and hence every cyclic subgroup is 2-subnormal, it is enough to show that $(ii)$ implies $(iii)$.

Let $a,b\in G$ and take $x\in G\backslash T_2(G)$ of infinite order. Lemma \ref{infinite} implies that $x\in R_2(G)$ and there exist $m,n\geq 0$ such that $x^m a,x^n b\in R_2(G)$. But $R_2(G)$ is a subgroup of $G$ \cite{Ka} (see also \cite[Corollary 1, p. 44]{Ro2}), so that $x^m$ is also an element of $R_2(G)$. On the other hand, by Lemma \ref{infinite}, $\l a,b,x\r=\l x^ma,x^nb,x\r$ is nilpotent of class $\leq 3$. Thus we get
$$1=[x^m a, b,b]=[x^m, b,b][x^m, b,a,b][a,b,b]=[a,b,b].$$
This proves that $G$ is 2-Engel.
\end{proof}

Turning now to the class of torsion groups which are generalized 2-Baer groups, our characterization is parallel to the one of generalized 1-Baer groups \cite{Ca}.

\begin{thm}\label{main} Let $G$ be a torsion group and assume that $G$ is a generalized $2$-Baer group. Then $G=H\times K$ where $H$ is a locally finite $p$-group in $\B$ for some prime $p$ and $K$ is a $2$-Baer group without elements of order $p$.
\end{thm}

\begin{proof}
By Proposition \ref{gen}, the group $G$ is locally nilpotent, hence $G$ is the direct product of its Sylow $p$-subgroups.
Let $H_p$ be a Sylow $p$-subgroup of $G$ and suppose $T_2(H_p)=H_p$. Thus $G=H_p\times K_{p'}$, where $K_{p'}$ is the complement of $H_p$ in $G$. Let $h\in H_p$ and $k\in K_{p'}$, with $\l h\r\ntriangleleft_2 H_p$. Clearly, if $\langle hk\rangle\lhd_2 G$, then $\langle h\rangle=\langle hk\rangle\cap H_p\lhd_2 H_p$. Therefore, we have necessarily $hk\in T_2(G)$. It follows that $H_p K_{p'}=T_2(H_p)K_{p'}\leq T_2(G)$ and $G=T_2(G)$, a contradiction. This means that $H_{p}\in \B$ for all primes $p$. On the other hand, if $T_2(K_{p'})\neq 1$, we get as before $H_p\leq H_p T_2(K_{p'})\leq T_2(G)$ and $G=T_2(G)$. Hence there exists a prime $p$ such that $G=H\times K$ where $H$ is a $p$-group in $\B$ and $K$ is a 2-Baer group without elements of order~$p$.
\end{proof}

\begin{cor}
A group $G$ is a generalized $T_2$-group if and only if $G=H\times K$, where $H$ is a locally finite $p$-group in $\T$ for some prime $p$ and $K$ is a $2$-Baer torsion group without elements of order $p$.
\end{cor}

\begin{proof}
Assume $G=H\times K$, where $H$ is a $p$-group and $K$ is a torsion group without elements of order $p$. Let $x=hk$, with $h\in H$ and $k\in K$, such that $\l x\r\ntriangleleft \l x^G\r$. Then $\langle h\rangle\ntriangleleft \l h^H\r$, because otherwise $x^{(x^{h'k'})}\in \langle h,k\rangle=\langle x\rangle$ for all $h'\in H$ and $k'\in K$. So $h\in T_2(H)$ and $x\in T_2(H)K$. This proves that $T_2(G)\leq T_2(H)\times K$. Similarly we have $T_2(G)\leq H\times T_2(K)$. Thus, if $G$ is a generalized $T_2$-group, the claim follows from Theorems \ref{nontorsion} and \ref{main}. The converse is trivial.
\end{proof}

Let $c$ and $d$ be positive integers. A group $G$ is said to be a {\em group of type $(d,c)$} if every $d$-generator subgroup of $G$ is nilpotent of class at most $c$. In this context we are interested in groups of type $(d,3d-3)$ and of type $(d,cd)$, that were investigated by Endimioni in \cite{En91,En93} and \cite{En89}, respectively.

\begin{thm}\label{End}
Let $G$ be a locally finite $p$-group and assume that $G$ is a generalized $2$-Baer group.
\begin{itemize}
\item[$(i)$] If $p\neq 5$, then $G$ is soluble.
\item[$(ii)$] If $p>5$, then $G$ is nilpotent of bounded class.
\item [$(iii)$] If $p=5$ and $G$ is soluble, then $G$ is nilpotent of class bounded by a function depending on the derived length.
\end{itemize}
\end{thm}

\begin{proof}
By Proposition \ref{gen}, the group $G$ is of type $(d,3d-3)$ with $d=5$, and of type $(d,3d)$ with $d=2$. In particular $G$ is $6$-Engel. If $p\neq 2,5$ then $G$ is soluble of bounded derived length, by \cite[Proposition 3(ii)]{En91}. On the other hand, if $p=2$, then $G$ is again soluble, as shown in \cite[p. 2830]{AT} (see also \cite{En93}). However, when $p>5$, the main theorem of \cite{En91} guarantees that $G$ is nilpotent of bounded nilpotency class. Finally, if $p=5$ and $G$ is soluble, we get $(iii)$ by \cite[Corollary 2]{En89}.
\end{proof}

\section{Metabelian generalized $T_2$-groups}

As Theorem \ref{End} of the last section suggests, there might be a universal bound for the nilpotency class of $p$-groups in $\T$ provided $p>5$. But so far we have not been able to find an explicit bound for their nilpotency class. However in the case of metabelian groups we have been able to show that $p$-groups in $\T$ have nilpotency class 3, provided that $p>2$. In the last section we will provide an example of a $2$-group in $\T$ of nilpotency class 4, showing that our restriction to odd primes is necessary.

Let $G$ be a metabelian group and denote by $G'$ its commutator subgroup. For integers $m,n$ with $n>0$ and for all $x,y,z\in G, c\in G'$, in the sequel we will use freely the identities
$$[c,x,y]=[c,y,x],$$
$$[xy,_n z]=[x,_n z][x,_n z,y][y,_n z]$$
and, if $G$ is nilpotent of class 3,
$$[x^m,y,z]=[x,y^m,z]=[x,y,z^m]=[x,y,z]^m.$$

Now we are ready to state and prove the main result of this section.

\begin{thm}\label{metp}
Let $G$ be a metabelian $p$-group with $p$ an odd prime. If $G\in \T$, then $G$ is nilpotent of class $3$.
\end{thm}

\begin{proof}
By Theorem 7.36 $(i)$ in \cite{Ro2}, it is enough to show that $G$ is a 3-Engel group. For this it suffices to show that for every pair of elements $x,y$ in $G$ we have $[x,_3 y]=1$ and $[y,_3 x]=1$. We consider three cases: $x$ and $y$ are both in $G\backslash T_2(G)$, only one of the elements is in $G\backslash T_2(G)$, and $x$ and $y$ are both in $T_2(G)$. Set $H=\l x,y\r$.

In the first case, we observe by Lemma \ref{closure} that $[y,_3 x]=1$ and $[x,_3 y]=1$. Now consider the second case, where we can assume without loss of generality that $x\in G\backslash T_2(G)$ and $y\in T_2(G)$. We observe that also $xy\in G\backslash T_2(G)$ and hence $\l x^G\r\l (xy)^G\r$ is a group of nilpotency class 4 at most and $H\leq \l x^G\r\l (xy)^G\r$. By Lemma \ref{closure} we have $[y,_3 x]=1$ and $[x,_3 xy]=1$. To prove our claim, we need to show $[x,_3 y]=1$. Since $H$ is of class 4 at most, we obtain by linear expansion that
\begin{equation}
1=[x,_3 xy]=[x,_3 y][x,y,y,x][x,y,x,y]. \tag{3.1.1}
\end{equation}
Substituting $x^{-1}$ for $x$ into (3.1.1) and again expanding linearly yields
\begin{equation}\label{x-1}
1=[x^{-1},_3 x^{-1}y]=[x,_3 y]^{-1}[x,y,y,x][x,y,x,y]. \tag{3.1.2}
\end{equation}
Multiplying (3.1.1) by the inverse of (3.1.2) leads to $[x,_3 y]^2=1$. Since $G$ has no involutions, we conclude $[x,_3 y]=1$, the desired result.

Before we prove the third case, i.e. $x$ and $y$ are both in $T_2(G)$, we want to show that $H$ is nilpotent of class 3 at most, provided that not both $x$ and $y$ are in $T_2(G)$. Now (3.1.1) together with $[x,_3 y]=1$ yields that
$$1=[x,y,y,x][x,y,x,y].$$
By III, 2.12 (b) in \cite{Hu}, we have that any 2-generator group of class 4 is metabelian and thus $[x,y,y,x]=[x,y,x,y]$. This together with the fact that our group has no involutions yields $1=[x,y,y,x]$. Now we observe that
$$\gamma_4(H)=\l[y,_3 x], [x,_3 y], [x,y,y,x]\r.$$
But we showed that each of the generators is the identity in $H$. We conclude $\gamma_4(H)=1$ and $H$ is of class 3.

To prove the remaining case, i.e. $x,y\in T_2(G)$ implies $[x,_3 y]=[y,_3 x]=1$, we recall that $G$ is metabelian. Let $z\in G\backslash T_2(G)$ and $x,y\in T_2(G)$. Obviously, $xz\in G\backslash T_2(G)$. Then the subgroups $\l y,z\r$ and $\l y,xz\r$ are nilpotent of class $3$, by the previous argument. It follows that
$$1=[xz,_3 y]=[x,_3 y][x,_3 y, z][z,_3 y]=[x,_3 y][x,_3 y, z].$$
By commuting with $z$, we get $1=[x,_3 y,z][x,_3 y,z,z]$. Commuting again with $z$ yields $1=[x,_3 y,z,z]$. Thus $1=[x,_3 y,z]$ and consequently $1=[x,_3 y]$. Similarly we obtain $[y,_3 x]=1$. We conclude that $G$ is 3-Engel.
\end{proof}

\section{Examples and Counterexamples}

In this section we provide various examples and counterexamples in support of claims made earlier in this paper. In Theorem \ref{metp} it was shown that metabelian $p$-groups in $\T$, $p$ an odd prime, have nilpotency class exactly 3. Our first example shows that this result cannot be extended to $2$-groups. Our claims have been verified by {\sf GAP} \cite{GAP}.

\begin{exa}\label{2-group}
Let
$$G=\l x,y\,|\,x^{16}=y^{16}=1, (xy^{-1})^2=[x,y]^4=1, [x,y,x]=x^4, [x,y,y]=y^4\r.$$
Then $G$ is a metabelian group of order $2^7$ and nilpotency class $4$. Moreover $|T_2(G)|=2^6$, thus $G$ is a generalized $T_2$-group.
\end{exa}

For our next example we need the following expansion formula for metabelian groups which can be found in \cite{HK}.

\begin{lem}
Let $G$ be a metabelian group, $x\in G$ and $n$ an integer. Then
\begin{equation}
(xy^{-1})^n=x^n (\prod_{0<i+j<n}[x,_i y,_j x]^{\binom{n}{i+j+1}})y^{-n}. \tag{4.2.1}
\end{equation}
\end{lem}

The next example provides a $p$-group $G$ with $1\neq T_2(G)\neq G$ for any prime $p$.

\begin{exa}\label{ex}
For any prime $p$, let
$$G=\langle x,y\,|\,x^{p^3}=y^{p^3}=[x,y]^p=[x,y,x]=1, [x,y,y]=x^{p^2}=y^{p^2}\rangle.$$
Then $G$ is a generalized $T_2$-group of nilpotency class $3$ and of order $p^6$ with $|T_2(G)|=p^5$.
\end{exa}

\begin{proof}
We divide the proof into three steps. In the first one we will show that the group $G$ is nilpotent of class $3$ and of order $p^6$.

Let $N=\l [x,y], x^{p^2}\r$. We have $N\lhd G$ and $G/N$ abelian, so that $G'\leq N$. It follows that $G'=N\leq Z_2(G)$ and $G$ is nilpotent of class $3$ (see also \cite[Proposition 3.1]{ABT}). In particular, $G$ is finite and, by the Burnside Basis Theorem \cite[5.3.2]{Ro},
$$Frat(G)=G'G^p=\l x^p,y^p,[x,y]\r.$$
Thus every $g\in G$ can be written as $g=x^m y^n z$, where $0\leq m,n\leq p-1$ and $z\in Frat(G)$. Since $|Frat(G)|=p^4$, we deduce that $|G|=p^6$.

In our next step we will provide a necessary and sufficient condition for $g\in G$, such that $\l g\r$ is 2-subnormal in $G$, as follows: Let $g=x^m y^n z\in G$ with $0\leq m,n\leq p-1$ and $z\in Frat(G)$. Then
$\langle g\rangle\lhd_2 G$ if and only if $m\equiv n\equiv 0$ $(mod\,p)$ or $m+n\not\equiv 0$ $(mod\,p)$.

This can be seen as follows. Clearly, we have $Frat(G)\leq Z_2(G)$ and $exp(G')=p$. Then
$$[x,g,g]=[x,y^n,y^n]=x^{n^2 p^2} \quad{\rm and} \quad [y,g,g]=[y,x^m,y^n]=x^{-mn p^2}.$$
Moreover, applying (4.2.1), we get
$$g^{p^2}=x^{m p^2} y^{n p^2}=x^{(m+n)p^2}.$$
Hence $[x,g,g]\in\langle g\rangle$ and $[y,g,g]\in\langle g\rangle$
if and only if
$$x^{n^2 p^2}\in\langle x^{(m+n)p^2}\rangle\quad{\rm and}\quad
x^{-mn p^2}\in\langle x^{(m+n)p^2}\rangle.$$
This is equivalent to the existence of integers $\alpha$ and $\beta$ satisfying the following congruences:
\begin{equation}
\hspace{0.55cm}n^2 \equiv \alpha(m+n)\;(mod\,p); \tag{4.3.1}
\end{equation}
\begin{equation}
-mn \equiv \beta(m+n)\;(mod\,p). \tag{4.3.2}
\end{equation}

Now let $\langle g\rangle\lhd_2 G$. This gives $[x,g,g]\in\langle g\rangle$ and $[y,g,g]\in\langle g\rangle$. Therefore (4.3.1) and (4.3.2) are satisfied for some $\alpha$ and $\beta$. Assuming $m+n\equiv 0$ $(mod\,p)$, it follows that $m\equiv n\equiv 0$ $(mod\,p)$. Hence, either $m\equiv n\equiv 0$ $(mod\,p)$ or $m+n\not\equiv 0$ $(mod\,p)$.

Conversely, let $h=x^{m'} y^{n'} z'$ be an arbitrary element of $G$ where $0\leq m',n'\leq p-1$ and $z'\in Frat(G)$. It is easy to see that
$$[h,g,g]=[x^{m'}y^{n'},g,g]=[x^{m'},g,g][y^{n'},g,g]=[x,g,g]^{m'} [y,g,g]^{n'}.$$
The claim will follow once we have shown that $[h,g,g]\in \l g\r$. If $m\equiv n\equiv 0$ $(mod\,p)$, then $g\in Frat\,(G)$ and $[h,g,g]=1$. Suppose $m+n\not\equiv 0$ $(mod\,p)$.
In this case there always exist integers $\alpha$ and $\beta$ satisfying (4.3.1) and (4.3.2). Thus, by the above, $[x,g,g],[y,g,g]\in\langle g\rangle$ and so is $[h,g,g]$.

Next we will show that the group $G$ is a generalized $T_2$-group by establishing that $|T_2(G)|=p^5$. Let $g=x^m y^n z\in G$. Then, by the above, we have $\langle g\rangle\ntriangleleft_2 G$ if and only if $m+n\equiv 0\;(mod\,p)$ and $mn\not\equiv 0\;(mod\,p)$.
It follows
$$T_2(G)=\langle g\,|\,g=x^m y^{p-m} z, 0<m\leq p-1, z\in Frat\,(G)\rangle.$$
On the other hand, by (4.2.1), we have $x^m y^{p-m} z=(xy^{p-1})^m z'$ for a suitable $z'\in Frat\,(G)$. It follows that
$$T_2(G)=\langle xy^{p-1},x^p,y^p,[x,y]\rangle.$$
Notice also that $xy^{p-1}\notin Frat(G)$, otherwise $1=[xy^{p-1},y,y]$ and consequently $1=[x,y,y]$.
However, again by (4.2.1), it follows
$$(xy^{p-1})^p=x^p[x,y^{1-p}]^{\binom{p}{2}}[x,y^{1-p},y^{1-p}]^{\binom{p}{3}}y^{(p-1)p}\in Frat\,(G)$$
and
$$(xy^{p-1})^{p^2}=x^{p^2}y^{(p-1)p^2}=x^{p^2}x^{(p-1)p^2}=x^{p^3}=1.$$
Since $|Frat(G)|=p^4$, we conclude that $|T_2(G)|=p^{5}$. This proves that $G\in\T$.
\end{proof}

\end{document}